\documentclass[preprint,3pt]{elsarticle}

\usepackage{amssymb}
\usepackage{amsthm,amssymb,amsmath,extarrows}
\usepackage{pdfpages}
\usepackage{graphicx}
\usepackage{geometry}
\usepackage{url}
\usepackage{enumerate}
\usepackage{extarrows}
\theoremstyle{plain} \newtheorem{thm}{\bf Theorem}[section]
\newtheorem{corollary}[thm]{\bf Corollary} \newtheorem{lem}[thm]{\bf Lemma}
 \newtheorem{defn}[thm]{\bf Definition} \theoremstyle{remark}

\newtheorem{rem}{Remark}

\newcommand{\Rg}       {{\hbox{I\kern-.22em\hbox{R}}}}
\newcommand{\Pg}       {{\hbox{I\kern-.22em\hbox{P}}}}
\newcommand{\Eg}       {{\hbox{I\kern-.22em\hbox{E}}}}

\definecolor{lw}{RGB}{0,0,255}

\journal{journal}

\begin{document}

\begin{frontmatter}



\title{Maximum Likelihood Estimation for Mixed Fractional Vasicek Processes}


\author{*Chunhao Cai\\
*Corresponding author\\
School of Mathematics, Shanghai University of Finance and Economics, Shanghai, China.\\
caichunhao@mail.shufe.edu.cn\\
Yinzhong Huang\\
School of Mathematics, Shanghai University of Finance and Economics, Shanghai, China.\\
1307111968@qq.com\\
Weilin Xiao\\
School of Management, Zhejiang University, Zhejiang, China.\\
wlxiao@zju.edu.cn}

\begin{abstract}
The mixed fractional Vasicek model, which is an extended model of the traditional Vasicek model, has been widely used in modelling volatility, interest rate and exchange rate. Obviously, if some phenomenon are modeled by the mixed fractional Vasicek model, statistical inference for
this process is of great interest. Based on continuous time observations, this paper considers the problem of estimating the drift parameters in the mixed fractional Vasicek model. We will propose the maximum likelihood estimators of the drift parameters in the mixed fractional Vasicek model with the Radon-Nikodym derivative for a mixed fractional Brownian motion. Using the fundamental martingale and the Laplace transform, both the strong consistency and the asymptotic normality of the maximum likelihood estimators have been established for all $H\in(0,1)$, $H\neq 1/2$.

\end{abstract}


\end{frontmatter}

\section{Introduction}\label{introdution}

The standard Vasicek models, including the diffusion
models based on the Brownian motion and the jump-diffusion models driven by L\'{e}vy processes, provide good service in cases where the data demonstrate the Markovian property and the lack of memory. However, over the past few decades, numerous empirical studies have found that the phenomenon of long-range dependence may observe in data of hydrology, geophysics, climatology and telecommunication, economics and finance. Consequently,
several time series models or stochastic processes have been proposed to capture long-range dependence, both in discrete time and in continuous time. In the continuous time case, the best known and widely used stochastic process that
exhibits long-range dependence or short-range dependence is of course the fractional Brownian motion (fBm), which describes the degree of dependence by the Hurst parameter. This naturally explains the appearance of fBm in the modeling of some properties of \textquotedblleft Real-world\textquotedblright\ data. As well as in the diffusion model with a fBm, the mean-reverting property is very attractive to understand volatility modeling in finance. Hence, the fractional Vasicek model (fVm) becomes the usual candidate to capture some phenomena of the volatility of financial assets (see, for example, \cite{comte1998,aitsahlia2008,gatheral2018}). More precisely,
fVm can be described by the following Langevin equation
\begin{equation}\label{eq: fvasicekmodel}
dX_{t}=\left(\alpha-\beta X_{t}\right)dt+\gamma dB_{t}^{H},\,\,\, t\in [0,T],
\end{equation}
where $\beta,\gamma \in \mathbb{R}^{+}$, $\alpha \in \mathbb{R}$, the initial condition is set at $X_{0}=0$, and $B_{t}^{H}$, a fBm
with Hurst parameter $H\in (0,1)$, is a zero mean Gaussian process with the covariance
\begin{equation}
\mathbb{E}\left( B_{t}^{H}B_{s}^{H}\right) =R_{H}(s,t)=\frac{1}{2}\left(
|t|^{2H}+|s|^{2H}-|t-s|^{2H}\right) \,.  \label{eq: cov fBm}
\end{equation}%

The process $B_{t}^{H}$ is self-similar in the sense that $\forall a\in
R^{+} $, $B_{at}^{H}\overset{d}{=}a^{H}B_{t}^{H}$. It becomes the standard Brownian motion $W_{t}$ when $H=1/2$ and can be represented as a stochastic integral with respect to standard Brownian motion. When $1/2<H<1$, it has long-range dependence in
the sense that $\sum_{n=1}^{\infty}\mathbb{E}%
\left(B_{1}^{H}(B_{n+1}^{H}-B_{n}^{H})\right)=\infty $. In this case, the positive (negative) increments are likely to be followed by positive (negative) increments. The parameter $H$ which is also called the self similarity parameter, measures the intensity of the long range dependence. Recently, borrowing the idea of \cite{Cheridito}, these papers \cite{li2017efficient,jacod2018limit} used the mixed fractional Vasicek model (mfVm) to describe the  some phenomena of the volatility of financial assets, which can be expressed as
\begin{equation}\label{eq: mixed vasicek}
dX_t=(\alpha-\beta X_t)dt+\gamma d\xi_t,\, t\in [0,T],\, X_0=0.
\end{equation}
where $\beta,\gamma \in \mathbb{R}^{+}$, $\alpha \in \mathbb{R}$, the initial condition is set at $X_{0}=0$, $\xi_t=W_t+B_t^H,\, H\in (0,1)$ is a mixed fractional Brownian motion defined in the paper of \cite{CCK}.

When the long term mean $\alpha$ in \eqref{eq: mixed vasicek} is known (without loss of generality, it is assumed to be zero), \eqref{eq: mixed vasicek} becomes the mixed fractional Ornstein-Uhlenbeck process (mfOUp). Using the canonical representation and spectral structure of mfBm, the authors of \cite{CK18} originally proposed the maximum likelihood estimator (MLE) of $\beta$ in \eqref{eq: mixed vasicek} and considered the asymptotical theory for this estimator. With the Laplace Transform and the limit presence of the eigenvalues of the covariance operator for the mfBm, the authors of \cite{CK18b} obtained accurate asymptotic approximations for the eigenvalues and the eigenfunctions for mfVm. Using an asymptotic  approximation for the eigenvalues of its covariance operator, the paper of \cite{CKM20} explained mfVm from the view of spectral theory. Some surveys and complete literatures related to the parametric and other inference procedures for stochastic models driven by mfBm was summarized in a recent monograph of \cite{mishura2018book}. However, in some situation the long term mean $\alpha$ in \eqref{eq: mixed vasicek} is always unknown. Thus, it is important to estimate all the drift parameters, $\alpha$ and $\beta$ in mfVm.   To the best of our knowledge, the asymptotic theory of MLE of $\alpha$ and $\beta$ have not developed yet. This paper will fill in the gaps in this area. Using the Girsanov formula for mfBm, we introduce the MLE for both $\alpha$ and $\beta$. When a continuous record of observations of $X_t$ is available, both the strong consistency and the asymptotic laws of MLE are established in the stationary case for the Hurst parameter $H\in(0,1)$.

The rest of the paper is organized as follows. Section \ref{preliminary} introduces some preliminaries of mfBm. Section \ref{estimators} proposes the MLE for the drift parameters in the mfVm and studies the asymptotic properties of MLE for the Hurst parameter range $H\in(0,1)$ in the stationary case. Section \ref{proof main results} provides the proofs of the main results of this paper. Some technical lemmas are gathered in the Appendix. We use the following notations throughout the paper: $\overset{a.s.}{\rightarrow }$, $\overset{\mathbf{P}}{\rightarrow }$, $\overset{d}{\rightarrow }$ and $\sim $ denote convergence almost surely, convergence in probability, convergence in distribution, and asymptotic equivalence, respectively, as $T\rightarrow \infty $.

\section{Preliminaries}\label{preliminary}

This section is dedicated to some notions that are used in our paper, related mainly to the integro-differential equation and the Radon-Nikodym derivative of mfBm. In fact, mixtures of stochastic processes can have properties, quite different from the individual components. The mfBm drew considerable attention since some of its properties have been discovered in \cite{Cheridito,CCK,mishura2018book}. Moreover, the mfBm has been proved useful in mathematical finance (see, for example, \cite{Cheridito2003}). We start by recalling the definition of the main process of our work, which is the mfBm. For more details about this process and its properties, the interested reader can refer to \cite{Cheridito,CCK,mishura2018book}.

\begin{defn}\label{definition mfBm}
A mfBm of the Hurst parameter $H\in(0,1)$ is a process $\xi=(\xi_t,\, t\in [0,T])$ defined on a probability space $(\Omega, \mathcal{F}, \mathbf{P})$ by
$$
\xi_t=W_t+B_t^H
$$
where $W=(W_t,\, t\in [0,T])$ is the standard Brownian motion and $B^H=(B_t^H,\, t\in [0,T])$ is the independent fBm with the Hurst exponent $H\in (0,1)$ and the covariance function
$$
K(s,t)=\mathbf{E}B_t^HB_s^H=\frac{1}{2}\left(t^{2H}+s^{2H}-|t-s|^{2H-1}\right).
$$
\end{defn}

Let us observe that the increments of mfBm are stationary and
$\xi_t$ is a centered Gaussian process with the covariance function
$$\mathbf{E}\left[\xi_{t}^{H} \xi_{s}^{H}\right]=\min \{t, s\}+\frac{1}{2}\left[t^{2 H}+s^{2 H}-|t-s|^{2 H}\right], \quad s, t \geqslant 0\,.$$

In particular, for $H>1/2$, increments of mfBm exhibit long-range dependence, which makes it important in modeling volatile in finance. Let $\mathcal{F}^{\xi}=(\mathcal{F}_t^{\xi},\, t\in [0,T])$. We will use the canonical representation suggested in \cite{CCK}, based on the martingale
$$
M_t=\mathbf{E}(B_t|\mathcal{F}_t^{\xi}),\,\,\, t\in [0,T].
$$

It is clear that the process $M$ is a $\mathcal{F}^{\xi}$-martingale and admits the following representation
\begin{equation*}
M_t=\int_0^t g(s,t)d\xi_s ,\,\, \langle M\rangle_t=\int_0^t g(s,t)ds,\, t\geq 0,\,\,\, t\in [0,T],
\end{equation*}
where the stochastic integral is defined for $L^2(0, T)$ deterministic integrands in the usual way and the kernel $g(s,t)$ solves the integro-differential equation
$$
g(s,t)+H\frac{d}{ds}\int_0^t g(r,t)|r-s|^{2H-1}\textrm{sign}(s-r)dr=1,\,\,\, 0<s\neq t\leq T.
$$

For $H>1/2$, the equation $g(s,t)$ is a Wiener-Hopfner equation:
$$
g(s,t)+H(2H-1)\int_0^t g(r,t)|r-s|^{2H-2}dr=1,\,\, 0\leq s\leq t\leq T.
$$

Moreover, from \cite{CCK} we have the following result.

\begin{lem}
For $H>1/2$, the quadratic variation of the $\mathcal{F}^{\xi}$-martingale $M$ is
\begin{equation}\label{eq: innovation}
\langle M\rangle_t=\int_0^t g^2(s,s)ds
\end{equation}
and moreover,
\begin{equation}\label{eq: G}
\xi_t=\int_0^t G(s,t)dM_s,\,\, t\in [0,T]
\end{equation}
where G(s,t) is defined in the equation (2.20) of \cite{CCK}.
\end{lem}

The equality \eqref{eq: innovation} suggests that the martingale $M$ admits innovation type representation, which can be used to analyse the structure of the mfBm with stochastic drift. Let us mention that the innovation type representation of the martingale $M$ can be also used to derive an analogue of Girsanov's theorem, which will be the key tool for constructing  MLE.
\begin{corollary}
Consider a process $Y=(Y_t,\, t\in [0,T])$ defined by
$$
Y_t=\int_0^tf(s)ds+\xi_t,\,\,\, t\in [0,T]\,,
$$
where $f=(f(t),\, t\in [0,T])$ is a process with continuous path and $\mathbf{E}\int_0^T |f(t)|dt<\infty$, adapted to a filtration $\mathcal{G}=(\mathcal{G}_t)$ with respect to a martingale $M$. Then $Y$ admits the following representation
$$
Y_t=\int_0^t G(s,t)dZ_s
$$
with $G(s,t)$ defined in \eqref{eq: G} and the process $Z=(Z_t,\, t\in [0,T])$ can be written as
$$
Z_t=\int_0^t g(s,t)dY_s,\,\,\, t\in [0,T]\,.
$$

Let us mention that $Z_t$ is a $\mathcal{G}$-martingale with Doob-Meyer decomposition
$$
Z_t=M_t+\int_0^t \Phi(s)d\langle M\rangle_s\,,
$$
where
$$
\Phi(t)=\frac{d}{d\langle M\rangle_t}\int_0^t g(s,t)f(s)ds.
$$

In particular, $\mathcal{F}_t^Y=\mathcal{F}_t^{Z},\, \mathbf{P}-a.s. $ for all $t\in [0,T]$. Moreover, if
$$
\mathbf{E}\exp \left\{ -\int_0^T\Phi(t)dM_t-\frac{1}{2}\int_0^T \Phi^2(t)d\langle M\rangle_t \right\}=1,
$$
then the measures $\mu^{\xi}$ and $\mu^Y$ are equivalent and the corresponding Radon-Nikodym derivative is given by
$$
\frac{d\mu^Y}{d\mu^{\xi}}(Y)=\exp \left\{\int_0^T \hat{\Phi}(t)dZ_t-\frac{1}{2}\int_0^T \hat{\Phi}^2(t)d\langle M\rangle_t          \right\},
$$
where $\hat{\Phi}(t)=\mathbf{E}(\Phi(t)|\mathcal{F}_t^Y)$.
\end{corollary}

\section{Estimators and asymptotic behaviors}\label{estimators}
Let us define
\begin{equation}\label{eq: change model}
Z_t=\int_0^t g(s,t)dX_s,\, \,\, Q_t=\frac{d}{d\langle M\rangle_t}\int_0^t g(s,t)X_sds,\,\, t\in [0,T]\,.
\end{equation}

Then using the quadratic variation of $Z$ on [0,T], we can estimate $\gamma$ almost surely from any small interval as long as we have a continuous observation of the process. Moreover, the estimation of $H$ in the mfBm have been done in \cite{MSD15}. As a consequence, for further statistical analysis, we assume that $H$ and $\gamma$ are known and, without loss of generality, from now on we suppose that $\gamma$ is equal to one.
For $\gamma=1$, our observation will be $Z=(Z_t,\, t\in [0,T])$, where $Z_t$ satisfies the following equation:
\begin{equation}\label{eq: Zt}
dZ_t=(\alpha-\beta Q_t)d\langle M\rangle_t+dM_t,\, t\in [0,T].
\end{equation}

Applying the analog of the Girsanov formula for a mfBm, we can obtain the following likelihood ratio and the explicit expression of the likelihood function:
\begin{equation}\label{eq: likelihood function}
\mathcal{L}_T(\alpha, \beta, Z^T)=\exp\left(\int_0^T (\alpha-\beta Q_t)  dZ_t-\frac{1}{2}\int_0^T  (\alpha-\beta Q_t)^2d\langle M\rangle _t \right)\,.
\end{equation}

\subsection{Only one parameter is unknown}
Denote the log-likelihood equation by $\Lambda(Z^T)=\log \mathcal{L}_T(\alpha,\beta,Z^T)$. First of all, if we suppose $\alpha$ is known and $\beta>0$ is the unknown parameter, then the MLE $\tilde{\beta}_T$ is defined by
\begin{equation}\label{eq: tilde beta}
\tilde{\beta}_T=\frac{\int_0^T \alpha Q_t d\langle M\rangle_t-\int_0^T Q_tdZ_t}{\int_0^T Q_t^2d\langle M\rangle_t}
\end{equation}
then using \eqref{eq: Zt} for all $H\in (0,1), H\neq 1/2$, the estimator error can be presented by
\begin{equation}\label{eq: tilde beta error}
\tilde{\beta}_T-\beta=\frac{\int_0^T Q_t dM_t}{\int_0^T Q_t^2 d\langle M\rangle_t}.
\end{equation}
We have the following results:
\begin{thm}\label{th: asymptotical tilde beta}
For $H>1/2$, 
$$
\sqrt{T}(\tilde{\beta}_T-\beta)\xrightarrow{d}\mathcal{N}(0,2\beta)
$$
and for $H<1/2$,
$$
\sqrt{T} (\tilde{\beta}_T-\beta)\xrightarrow{d} \mathcal{N}\left(0, \frac{2\beta^2}{2\alpha^2+\beta}\right)
$$
\end{thm}
Now we suppose $\beta$ is known and $\alpha$ is the parameter to be estimated. Then the MLE $\tilde{\alpha}_T$ is 
\begin{equation}\label{eq: tilde alpha}
\tilde{\alpha}_T=\frac{Z_T+\beta\int_0^T Q_t d\langle M\rangle_t}{\langle M\rangle_T}
\end{equation}
Still with \eqref{eq: Zt}, the estimator error will be 
\begin{equation}\label{asymptotical tilde alpha}
\tilde{\alpha}-\alpha=\frac{M_T}{\langle M\rangle_T}.
\end{equation}
The asymptotical property is the same as well as the linear case which has been demonstrated in \cite{CK18}. That is for $H>1/2$, $T^{1-H}(\tilde{\alpha}_T-\alpha)\xrightarrow{d} \mathcal{N}(0,v_H)$
where $v_H$ is a constant defined in Theorem \ref{th: convergence normal} and for $H<1/2$, $\sqrt{T}(\tilde{\alpha}-\alpha)\xrightarrow{d}\mathcal{N}(0,1)$.

\subsection{Two parameters unknown}

Then taking the derivatives of the log-likelihood function, $\Lambda(Z^T)$, with respect to $\alpha$ and $\beta$ and setting them to zero, we can obtain the following results:
\begin{equation}\label{eq: partial derivative}
\left\{
\begin{array}{l}
\frac{\partial \Lambda(Z^T)}{\partial \alpha}=Z_T-\alpha \langle M\rangle_T+\beta \int_0^T Q_td\langle M\rangle_t=0\\
\frac{\partial \Lambda(Z^T)}{\partial \beta}=-\int_0^T Q_tdZ_t+\alpha \int_0^T Q_td\langle M\rangle_t-\beta\int_0^T Q_t^2d\langle M\rangle_t\\
\end{array}
\right.
\end{equation}

The MLE $\hat{\alpha}_T$ and $\hat{\beta}_T$ is solution of equation of \eqref{eq: partial derivative} and the maximization can be confirmed when we check the second partial derivative of $\Lambda(Z^T)$ by the Cauchy-Schwarz inequality. Now the solution of \eqref{eq: partial derivative} gives us:
$$
\hat{\alpha}_T=\frac{\int_0^T Q_tdZ_t\int_0^TQ_td\langle M\rangle_t-Z_T\int_0^TQ_t^2d\langle M\rangle_t}{\left(\int_0^T Q_td\langle M\rangle_t \right)^2-\langle M\rangle_T\int_0^T Q_t^2d\langle M\rangle_t}
$$
and
$$
\hat{\beta}_T=\frac{\langle M\rangle_T\int_0^TQ_tdZ_t-Z_T\int_0^TQ_td\langle M\rangle_t             }{ \left(\int_0^TQ_td\langle M\rangle_t \right)^2-\langle M\rangle_T\int_0^T Q_t^2d\langle M\rangle_t }.
$$

From the expression of $Z=(Z_t,\, t\in [0,T])$, we obtain that the error term of the MLE can be written as:
\begin{equation}\label{eq: Estimator error alpha} 
\hat{\alpha}_T-\alpha=\frac{\int_0^TQ_tdM_t\int_0^TQ_td\langle M\rangle_t-M_T\int_0^TQ_t^2d\langle M\rangle_t}{\left(\int_0^TQ_td\langle M\rangle_t\right)^2-\langle M\rangle_T\int_0^TQ_t^2d\langle M\rangle_t}
\end{equation}
and
\begin{equation}\label{eq: Estimator error beta}
\hat{\beta}_T-\beta=\frac{\langle M\rangle_T\int_0^TQ_tdM_t-M_T\int_0^TQ_td\langle M\rangle_t}{\left(\int_0^T Q_td\langle M\rangle_t\right)^2-\langle M\rangle_T\int_0^TQ_t^2d\langle M\rangle_t}.
\end{equation}

We can now describe the asymptotic laws of $\hat{\alpha}_T$ and $\hat{\beta}_T$ for $H\in (0,1)$ but $H\neq 1/2$.
\begin{thm}\label{th: convergence normal}
For $H>1/2$ and as $T\rightarrow\infty$, we have
\begin{equation}\label{eq: normality of beta}
\sqrt{T}\left(\hat{\beta}_T-\beta\right)\xrightarrow{d}\mathcal{N}\left(0,2\beta \right)\,,
\end{equation}
and
\begin{equation}\label{eq: normality of alpha}
T^{1-H}\left(\hat{\alpha}_T-\alpha\right)\xrightarrow{d}
\mathcal{N}\left(0,v_H\right)\,,
\end{equation}
where $v_H=\frac{2H\Gamma(H+1/2)\Gamma(3-2H)}{\Gamma(3/2-H)}.$
\end{thm}
\begin{thm}\label{th: convergence normal smaller}
In the case of $H<1/2$, the maximum likelihood estimator of $\hat{\beta}_T$ has the same property of asymptotical normality presented in \eqref{eq: normality of beta} and for $\hat{\alpha}_T$ we have 
\begin{equation}\label{eq: normality of alpha smaller}
\sqrt{T}(\hat{\alpha}_T-\alpha)\xrightarrow{d}\mathcal{N}(0,1+\frac{2\alpha^2}{\beta})
\end{equation}
\end{thm}
\begin{rem}
From the previous theorem, we can see that when $H>1/2$, whether one parameter is unknown or two parameters are unknown together, the asymptotical normality of the estimator error has the same result and, they are also the same as well as the linear case and Ornstein-Uhlenbeck process with pure fBm with Hurst parameter $H>1/2$. But for $H<1/2$, the situation changes, these differences comes from the limit representation of the quadratic variation of the martingale $M=(M_t,\, 0\leq t\leq T)$.
\end{rem}

Now we will consider the joint distribution of the estimator error. For $H<1/2$, if we consider $\vartheta=\left(\begin{array}{c}\alpha \\ \beta \end{array}\right)$ as the  two dimensional unknown parameter, then the following theorem gives us the joint distribution of the estimator error of $\hat{\vartheta}_T$:
\begin{thm}\label{th: joint normal}
The maximum likelihood estimator $\hat{\vartheta}_T=\left(\begin{array}{c}\hat{\alpha}_T \\ \hat{\beta}_T \end{array}\right)$ is asymptotically normal
\begin{equation}\label{eq: joint normal}
\sqrt{T}\left(\hat{\vartheta}_T-\vartheta\right)\xrightarrow{d} \mathcal{N}(\mathbf{0}, \mathcal{I}^{-1}(\vartheta)).
\end{equation}
where $\mathbf{0}=\left(\begin{array}{c}0 \\0\end{array}\right)$ and $\mathcal{I}(\vartheta)=\left(\begin{array}{cc}1 & -\frac{\alpha}{\beta} \\ -\frac{\alpha}{\beta} & \frac{1}{2\beta}+\frac{\alpha^2}{\beta^2}\end{array}\right)$ is the matrix of Fisher Information.
\end{thm} 

\begin{rem}
From theorem \ref{th: convergence normal smaller} we can see that the convergence rates of $\hat{\alpha}_T$ and $\hat{\beta}_T$ are the same and we can use the central limit theorem of martingale in the proof. On the contrary, when $H>1/2$ the convergence rates are different which causes a lot of difficulties and we leave it for further study.
\end{rem}

In the above discussions we concern on the asymptotical laws of the estimators however even in \cite{CK18} with $\alpha=0$, the authors have not considered the strong consistency of $\hat{\beta}_T$. In what follows, we conclude that $\hat{\beta}_T$ converges to $\beta$ almost surely.

\begin{thm}\label{thm: convergence strong}
For $H\in (0,1),\, H\neq 1/2$, the estimators of $\hat{\beta}_T$ is strong consistency, that is, as $T\rightarrow\infty$,
\begin{equation}
\hat{\beta}_T\xrightarrow{a.s.}\beta. 
\end{equation}
\end{thm}

\begin{rem}
For the estimator $\hat{\alpha}_T$, the strong consistency is clear when $\beta=0$ and the same proof for $\beta$ unknown, that is why we do not write this conclusion.
\end{rem}

\section{Proofs of the Main results}\label{proof main results}
\subsection{Proof of Theorem \ref{th: asymptotical tilde beta}}
From \eqref{eq: tilde beta error}, we have
$$
\sqrt{T}(\tilde{\beta}_T-\beta)=\frac{\frac{1}{\sqrt{T}}\int_0^T Q_t dM_t}{\frac{1}{T}\int_0^T Q_t^2d\langle M\rangle_t}.
$$
In fact the process $\left(\int_0^t Q_s dM_s,0\leq t\leq T\right)$ is a martingale and we can see that  when $H>1/2$,
$$
\frac{1}{T}\int_0^T Q_t^2d\langle M\rangle_t \xrightarrow {\mathbf{P}} \frac{1}{2\beta}
$$
and when $H<1/2$
$$
\frac{1}{T}\int_0^T Q_t^2 d\langle M\rangle_t\xrightarrow {\mathbf{P}}\left(\frac{\alpha}{\beta}   \right)^2+\frac{1}{2\beta}.
$$
The central limit theorem of martingale (see \cite{HH80}) achieves the proof.

\subsection{Proof of Theorem \ref{th: convergence normal}}
First, we consider the asymptotical normality of $\hat{\beta}_T$. Using \eqref{eq: Estimator error beta}, we have
\begin{equation}\label{eq: 16}
\sqrt{T}\left(\hat{\beta}_T-\beta\right)=\frac{\frac{1}{\sqrt{T}}\int_0^T Q_tdM_t-\frac{M_T}{\langle M\rangle_T}\frac{1}{\sqrt{T}}\int_0^TQ_td\langle M\rangle_t}{\left(\frac{1}{\sqrt{T \langle M\rangle_T}}\int_0^TQ_td\langle M\rangle_t\right)^2-\frac{1}{T}\int_0^TQ_t^2d\langle M\rangle_t}\,,
\end{equation}
where $M_T$ is a centered Gaussian random variable with variation $\langle M\rangle_T$.  Using Lemma \ref{limit V} and Lemma \ref{limit Q U}, we can obtain
\begin{equation}\label{eq: 17}
\frac{M_T}{\langle M\rangle_T}\frac{1}{\sqrt{T}}\int_0^TQ_td\langle M\rangle_t\xrightarrow{\mathbf{P}} 0\,,
\end{equation}
and
\begin{equation}\label{eq: 18}
\left(\frac{1}{\sqrt{T\langle M\rangle_T}}\int_0^TQ_td\langle M\rangle_t\right)^2\xrightarrow{\mathbf{P}}0\,.
\end{equation}

Combining \eqref{eq: 16}-\eqref{eq: 18} with Lemma \ref{limit Qt}, we can obtain \eqref{eq: normality of beta}.

Now, we deal with the convergence of $\hat{\alpha}_T$. From \eqref{eq: Estimator error alpha}, we can easily have
\begin{equation}\label{eq: 19}
T^{1-H}\left(\hat{\alpha}_T-\alpha\right)=\frac{\frac{1}{\sqrt{T}} \int_0^T Q_tdM_t   \frac{T^{1-H}}{\langle M\rangle_T\sqrt{T}}\int_0^T Q_td\langle M\rangle_t-T^{1-H}\frac{M_T}{\langle M\rangle_T}\frac{1}{T}\int_0^T Q_t^2d\langle M\rangle_t                   }{ \left(\frac{1}{\sqrt{T\langle M\rangle_T}}\int_0^T Q_td\langle M\rangle_t\right)^2 -\frac{1}{T}\int_0^T Q_t^2d\langle M\rangle_t}\,.
\end{equation}

It is worth noting that
\begin{equation}\label{eq: 20}
\frac{1}{\sqrt{T}} \int_0^T Q_tdM_t   \frac{T^{1-H}}{\langle M\rangle_T\sqrt{T}}\int_0^T Q_td\langle M\rangle_t\xrightarrow{\mathbf{P}}0\,,
\end{equation}
and
\begin{equation}\label{eq: 21}
 \left(\frac{1}{\sqrt{T\langle M\rangle_T}}\int_0^T Q_td\langle M\rangle_t\right)^2 \xrightarrow{\mathbf{P}}0.
\end{equation}

Moreover, from \cite{CK18}, we can see that
\begin{equation}\label{eq: 22}
 T^{1-H}\frac{M_T}{\langle M\rangle_T}\xrightarrow{d}\mathcal{N}(0, v_H),
\end{equation}
where $ v_H=\frac{2H\Gamma(H+1/2)\Gamma(3-2H)}{\Gamma(3/2-H)}$.
Finally, combining \eqref{eq: 19}-\eqref{eq: 22}, we can obtain \eqref{eq: normality of alpha}.

\subsection{Proof of Theorem \ref{th: convergence normal smaller}}
For $\hat{\beta}_T$, let us relook at the equation \eqref{eq: 16} with $H<1/2$. First of all let us develop the denominator, 
\begin{eqnarray*}
\left(\int_0^T Q_td\langle M\rangle_t \right)^2&=&\left(\frac{\alpha}{\beta}\right)^2\langle M\rangle_T^2+\left(\frac{\alpha}{\beta}\right)^2\left(\int_0^T V(t)d\langle M\rangle_t  \right)^2+\left(\int_0^T Q_t^Ud\langle M\rangle_t \right)^2\\
&-&2\left(\frac{\alpha}{\beta}  \right)^2\langle M\rangle_T\int_0^T V(t)d\langle M\rangle_t+2\left(\frac{\alpha}{\beta}\right)\langle M\rangle_T\int_0^T Q_t^Ud\langle M\rangle_t\\
&-&2\left(\frac{\alpha}{\beta}\right)\int_0^T V(t)d\langle M\rangle_t\int_0^T Q_t^Ud\langle M\rangle_t
\end{eqnarray*}
where $V(t)$ and $Q_t^U$ are defined in Lemma \ref{lem Qt}. On the other hand
\begin{eqnarray*}
\int_0^T Q_t^2d\langle M\rangle_t&=&\left(\frac{\alpha}{\beta}\right)^2\langle M\rangle_T+\left(\frac{\alpha}{\beta}\right)^2\int_0^T V^2(t)d\langle M\rangle_t+\int_0^T (Q_t^U)^2d\langle M\rangle_t\\
&-&2\left(\frac{\alpha}{\beta}\right)^2\int_0^T V(t)d\langle M\rangle_t+2\left(\frac{\alpha}{\beta}\right)\int_0^T Q_t^Ud\langle M\rangle_t-2\left(\frac{\alpha}{\beta}\right)\int_0^T V(t)Q_t^Ud\langle M\rangle_t
\end{eqnarray*}
Consequently, we have
\begin{eqnarray*}
\frac{1}{T\langle M\rangle_T}\left(\int_0^T Q_td\langle M\rangle_t \right)^2-\frac{1}{T}\int_0^T Q_t^2d\langle M\rangle_t&=&\frac{1}{T\langle M\rangle_T} \left(\frac{\alpha}{\beta}\right)^2\left(\int_0^T V(t)d\langle M\rangle_t  \right)^2+\frac{1}{T\langle M\rangle_T}  \left(\int_0^T Q_t^Ud\langle M\rangle_t \right)^2\\
&-&\frac{2}{T\langle M\rangle_T}\left(\frac{\alpha}{\beta}\right)\int_0^T V(t)d\langle M\rangle_t\int_0^T Q_t^Ud\langle M\rangle_t-\frac{1}{T}\int_0^T (Q_t^U)^2d\langle M\rangle_t\\
&-&\frac{1}{T}\left(\frac{\alpha}{\beta}\right)^2\int_0^T V^2(t)d\langle M\rangle_t+\frac{2}{T}\left(\frac{\alpha}{\beta}\right)\int_0^T V(t)Q_t^Ud\langle M\rangle_t.
\end{eqnarray*}
We study it one by one.  From Lemma \ref{limit V}, Lemma \ref{quadratic smaller} and Lemma \ref{V asymptotic smaller} we have
$$
\frac{1}{T\langle M\rangle_T} \left(\frac{\alpha}{\beta}\right)^2\left(\int_0^T V(t)d\langle M\rangle_t  \right)^2\rightarrow 0,\,\,  \frac{1}{T}\left(\frac{\alpha}{\beta}\right)^2\int_0^T V^2(t)d\langle M\rangle_t\rightarrow 0,\,\,             T\rightarrow 0.
$$
From \cite{CK18}, we can easily obtain
$$
\frac{1}{T}\int_0^T (Q_t^U)^2 d\langle M\rangle_t \xrightarrow{\mathbf{P}}\frac{1}{2\beta}.
$$
using the Lemma \ref{lim Qt U Vt smaller} we have 
\begin{equation}\label{eq: denominator}
\frac{1}{T\langle M\rangle_T}\left(\int_0^T Q_td\langle M\rangle_t \right)^2-\frac{1}{T}\int_0^T Q_t^2d\langle M\rangle_t\xrightarrow {\mathbf{P}}\frac{1}{2\beta}.
\end{equation}

Now we consider the numerator, 
\begin{eqnarray*}
\frac{1}{\sqrt{T}}\int_0^TQ_tdM_t-\frac{M_t}{\langle M\rangle_T \sqrt{T}}\int_0^T Q_td\langle M\rangle_t&=&-\left(\frac{\alpha}{\beta}\right)\frac{1}{\sqrt{T}}\int_0^TV(t)dM_t+\frac{1}{\sqrt{T}}\int_0^T Q_t^UdM_t\\
&+&\frac{M_T}{\langle M\rangle_T\sqrt{T}}\int_0^T \left(\left(\frac{\alpha}{\beta}\right)V(t)-Q_t^U\right)d\langle M\rangle_t
\end{eqnarray*}
From the previous proof it is not difficult to show that:
$$
-\left(\frac{\alpha}{\beta}\right)\frac{1}{\sqrt{T}}\int_0^TV(t)dM_t\xrightarrow{\mathbf{P}}0,\, \frac{M_T}{\langle M\rangle_T\sqrt{T}}\int_0^T \left(\left(\frac{\alpha}{\beta}\right)V(t)-Q_t^U\right)d\langle M\rangle_t\xrightarrow{\mathbf{P}}0,\, T\rightarrow \infty.
$$
with the fact in \cite{CK18}:
$$
\frac{1}{\sqrt{T}}\int_0^T Q_t^UdM_t\xrightarrow{d}\mathcal{N}\left(0,\frac{1}{2\beta}\right),
$$
we have 
\begin{equation}\label{eq: numerator}
\frac{1}{\sqrt{T}}\int_0^TQ_tdM_t-\frac{M_t}{\langle M\rangle_T \sqrt{T}}\int_0^T Q_td\langle M\rangle_t\xrightarrow {d}\mathcal{N}\left(0,\frac{1}{2\beta}\right).
\end{equation}
Then combining the equation \eqref{eq: denominator} with \eqref{eq: numerator} it is easy to obtain
$$
\sqrt{T}\left(\hat{\beta}_T-\beta\right)\xrightarrow {d}\mathcal{N}(0, 2\beta).
$$
Now we will look at $\hat{\alpha}_T$. In fact
$$
\sqrt{T}(\hat{\alpha}_T-\alpha)=\frac{\frac{1}{\langle M\rangle_T\sqrt{T}}\int_0^TQ_tdM_t \int_0^T Q_td\langle M\rangle_t-\frac{1}{\sqrt{T}}\frac{M_T}{\langle M\rangle_T} \int_0^T Q_t^2d\langle M\rangle_t                      }{\left(\frac{1}{\sqrt{T \langle M\rangle_T}}\int_0^TQ_td\langle M\rangle_t\right)^2-\frac{1}{T}\int_0^TQ_t^2d\langle M\rangle_t}.
$$
We observe that the denominator is the same formula in $\beta$ and it adapt the equation \eqref{eq: denominator}, we only need to consider the numerator. From Lemma \ref{limit V} and Lemma \ref{lim Qt U Vt smaller} it is easy to know
\begin{equation}\label{eq: 1}
\frac{1}{\langle M\rangle_T}\int_0^T Q_t d\langle M\rangle_t\xrightarrow {a.s.} \frac{\alpha}{\beta}
\end{equation}
For the numerator, 
$$
\frac{1}{\langle M\rangle_T\sqrt{T}}\int_0^TQ_tdM_t \int_0^T Q_td\langle M\rangle_t=\frac{T}{\langle M\rangle_T}\left(\frac{1}{\sqrt{T}}\int_0^TQ_tdM_t \frac{1}{T}\int_0^T Q_td\langle M\rangle_t\right)
$$
With Lemma \ref{limit V}, Lemma \ref{quadratic smaller} and  Lemma \ref{V asymptotic smaller}
\begin{equation}\label{eq: minus alpha1}
\frac{1}{\langle M\rangle_T\sqrt{T}}\int_0^TQ_tdM_t \int_0^T Q_td\langle M\rangle_t-\left(\frac{\alpha}{\beta}   \right)^2\frac{M_T}{\sqrt{T}}\xrightarrow {d}\frac{\alpha}{\beta}\mathcal{N}\left(0,\frac{1}{2\beta}\right)
\end{equation}
on the other hand
\begin{equation}\label{eq: minus alpha2}
\frac{1}{\sqrt{T}}\frac{M_T}{\langle M\rangle_T} \int_0^T Q_t^2d\langle M\rangle_t -\left(\frac{\alpha}{\beta}\right)^2\frac{M_T}{\sqrt{T}}\xrightarrow{d}\frac{1}{2\beta}\mathcal{N}(0,1).
\end{equation}
The further study tell us that the two convergence in distribution of \eqref{eq: minus alpha1} and \eqref{eq: minus alpha2} comes from the term $\frac{1}{\sqrt{T}}\int_0^T Q_t^U dM_t$ and $\frac{\langle M\rangle_T}{\sqrt{T}}$, when $M=(M_t,\, 0\leq t \leq T)$ is a martingale, then these two terms are asymptotically independent and then from equation \eqref{eq: 1}, \eqref{eq: minus alpha1} and \eqref{eq: minus alpha2} we can easily obtain
$$
\sqrt{T}(\hat{\alpha}_T-\alpha)\xrightarrow {d} \mathcal{N}\left(0, \frac{2\alpha^2}{\beta}+1     \right).
$$

\subsection{Proof of Theorem \ref{th: joint normal}}
From equations \eqref{eq: Estimator error alpha} and \eqref{eq: Estimator error beta} we have 
$$
\hat{\vartheta}_T-\vartheta=\left(\begin{array}{c}\hat{\alpha}_T \\\hat{\beta}_T\end{array}\right)-\left(\begin{array}{c}\alpha \\\beta\end{array}\right)=\mathcal{Q}^{-1}_T\mathcal{R}_T
$$ 
where 
$$
\mathcal{R}_T=\left(\begin{array}{c}M_T \\   -\int_0^T Q_t dM_t   \end{array}\right),\,\, \mathcal{Q}_T =\left(\begin{array}{cc}\langle M\rangle_T & -\int_0^T Q_t d\langle M\rangle_t \\  -\int_0^T Q_t d\langle M\rangle_t   & \int_0^T Q_t^2 d\langle M\rangle_t \end{array}\right).
$$
We can see that $\mathcal{R}_t,\, 0\leq t \leq T$ is a martingale and $\mathcal{Q}_t$ is its quadratic variation. Strictly speaking, in order to use the central limit theorem for martingale (see \cite{HH80}) it is better we can compute the Laplace Transform for $\mathcal{Q}_T$ to achieve the proof, but when the quadratic formula of $\int_0^T \left(Q_t^U\right)^2d\langle M\rangle_t$ has been verified in \cite{CK18}, here we just study the asymptotical properties of every component of $\mathcal{Q}_T$.

First of all, from \cite{CK18} we know 
\begin{equation}\label{eq: joint 1}
\lim_{T\rightarrow \infty} \frac{1}{T}\langle M\rangle_T=1.
\end{equation}
On the other hand from Lemma \ref{lem Qt}, Lemma \ref{limit V} and Lemma \ref{lim Qt U Vt smaller} we have 
\begin{equation}\label{eq: joint 2}
\frac{1}{T}\int_0^T Q_t d\langle M\rangle_t\xrightarrow{\mathbf{P}}\frac{\alpha}{\beta}.
\end{equation}
At last from Lemma \ref{lem Qt}, Lemma \ref{limit V}, Lemma \ref{lim Qt U Vt smaller} and \cite{CK18} we have 
\begin{equation}\label{eq: joint 3}
\frac{1}{T}\int_0^T Q_t^2d\langle M\rangle_t\xrightarrow {\mathbf{P}}\frac{1}{2\beta}+\frac{\alpha^2}{\beta^2}
\end{equation}
The limits of \eqref{eq: joint 1}, \eqref{eq: joint 2} and \eqref{eq: joint 3} achieve the proof.

\begin{rem}
In fact, it is easy to calculate
$$
\mathcal{I}^{-1}(\vartheta)=\left(\begin{array}{cc}1+\frac{2\alpha^2}{\beta} & 2\alpha \\  2\alpha & 2\beta\end{array}\right)
$$
which indicates Theorem \ref{th: convergence normal smaller}.
\end{rem}

\subsection{Proof of Theorem \ref{thm: convergence strong}}
We will prove the central limit theorem of $\hat{\beta}_T$. For $\hat{\alpha}_T$, the proof is similar. To simplify notation, we first assume $\alpha=0$. Then, using the fact $\alpha=0$, we can write
$$
\hat{\beta}_T-\beta=\frac{\int_0^T Q_t^U dM_t}{\int_0^T \left(Q_t^U\right)^2d\langle M\rangle_t}\,.
$$

Due to the strong law of large numbers, to get the convergence almost surely, it suffices to prove
\begin{equation}\label{eq: to infinite quadratic}
\int_0^T \left(Q_t^U \right)^2d\langle M\rangle_t\xrightarrow{a.s.}\infty\,.
\end{equation}

From the Appendix of \cite{Maru16}, if we define
$$
\mathcal{K}_T(\mu)=\frac{1}{T}\log \mathbf{E}\exp\left( -\mu \int_0^T \left(Q_t^U\right)^2d\langle M\rangle_t  \right),
$$
then 
$$
\lim_{T\rightarrow \infty} \mathcal{K}_T(\mu)=\frac{\beta}{2}-\sqrt{\frac{\beta^2}{4}+\frac{\mu}{2}}\,,
$$
for all $\mu>-\frac{\beta^2}{2}$. When $\mu>0$, the limit of the Laplace transform can be written as
$$
\lim_{T\rightarrow \infty}\mathbf{E}\left(-\mu \int_0^T \left(Q_t^U\right)^2 d\langle M\rangle_t\right)=0\,,
$$
which achieves \eqref{eq: to infinite quadratic}. 

Now we turn to the case of $\alpha \neq 0$, in this situation, using \eqref{eq: Estimator error beta}, we have
$$
\hat{\beta}_T-\beta=\frac{\langle M\rangle_T\int_0^TQ_tdM_t}{\left(\int_0^T Q_td\langle M\rangle_t\right)^2-\langle M\rangle_T\int_0^TQ_t^2d\langle M\rangle_t}-
\frac{\frac{M_T}{T\langle M\rangle_T}\int_0^TQ_td\langle M\rangle_t}{\left(\frac{1}{\sqrt{T\langle M\rangle_T}}\int_0^T Q_td\langle M\rangle_t\right)^2-\frac{1}{T}\int_0^TQ_t^2d\langle M\rangle_t}
$$

For the first term of the above equation, we can write 
$$
\frac{\langle M\rangle_T\int_0^TQ_tdM_t}{\left(\int_0^T Q_td\langle M\rangle_t\right)^2-\langle M\rangle_T\int_0^TQ_t^2d\langle M\rangle_t}=\frac{1}{\frac{\left(\int_0^T Q_td\langle M\rangle_t\right)^2}{\langle M\rangle_T\int_0^TQ_tdM_t}-\frac{\int_0^TQ_t^2d\langle M\rangle_t}{\int_0^TQ_tdM_t}}\,.
$$

From the proof of Lemma \ref{limit Qt} and equation \eqref{eq: to infinite quadratic}, we see immediately that 
$$
\frac{\int_0^TQ_t^2d\langle M\rangle_t}{\int_0^TQ_tdM_t}\xrightarrow{} \infty.
$$

With the previous proofs, we obtain
$$
\frac{\left(\int_0^T Q_td\langle M\rangle_t\right)^2}{\langle M\rangle_T\int_0^TQ_tdM_t}
$$
is bounded. This shows that first term tends to $0$ almost surely as well as the second term. Hence $H>1/2$ we have the strong consistency. The result for $H<1/2$ can be proved with the same method and we will not present again.

\section{Appendix: Auxiliary Results}
This section contains some technical results needed in the proofs of the main theorems of the paper. First, we introduce an important result from \cite{CK18}, which is provided by the following Lemma.  
\begin{lem}\label{quadratic variation big}
For $H>1/2$, we have
$$
\frac{d}{dT}\langle M\rangle_T \sim T^{1-2H},\,\,\, \left(\frac{d}{dT}\log \frac{d}{dT}\langle M\rangle_T      \right)^2 \sim T^{-2},\,\,\, T\rightarrow \infty.
$$
\end{lem}
\begin{lem}\label{quadratic smaller}
For $H<1/2$, we have 
$$
\frac{d}{dT}\langle M\rangle_T \sim const.,\,\,\, \left(\frac{d}{dT}\log \frac{d}{dT}\langle M\rangle_T    \right)^2\sim T^{-2},\,\,\, T\rightarrow \infty.
$$
\end{lem}

The following Lemma shows the relationship between mfOUp and mfVm.

\begin{lem}\label{lem Qt}
Let $U=(U_t,\, 0\leq t\leq T)$ be a mfOUp with the drift parameter $\beta$:
\begin{equation}\label{eq: mixed OU}
dU_t=-\beta U_tdt+d\xi_t,\,\, t\in [0,T],\,\, U_0=0\,.
\end{equation}

Then, we have 
$$
X_t=\frac{\alpha}{\beta}-\frac{\alpha}{\beta}e^{-\beta t}+U_t,\,\, t\in [0,T]\,.
$$

Moreover, we have the development of $Q_t$ with
\begin{equation}\label{eq: Qt}
Q_t=\frac{\alpha}{\beta}-\frac{\alpha}{\beta}V(t)+Q_t^U\,,
\end{equation}
where
\begin{equation}\label{V and Qu}
V(t)=\frac{d}{d\langle M\rangle_t}\int_0^tg(s,t)e^{-\beta s}ds,\,\,\, Q_t^U=\frac{d}{d\langle M\rangle_t}\int_0^t g(s,t)U_sds\,.
\end{equation}
\end{lem}

Next, we present some limit results.

\begin{lem}\label{limit V}
For $H\in (0,1)$ and $H\neq 1/2$, as $T\rightarrow \infty$, we have
$$
\int_0^TV(t)d\langle M\rangle_t \sim Const. \,,
$$
for some constant $C$.
\end{lem}
\begin{proof}
The result is clear when from the definition we know $0\leq g(s,t)\leq 1$. 
\end{proof}

\begin{lem}\label{limit Q U}
For $H>1/2$, as $T\rightarrow \infty$, we have
$$
\frac{1}{\sqrt{T\langle M\rangle_T}}\int_0^TQ_t^Ud\langle M\rangle_t\xrightarrow{\mathbf{P}}0\,.
$$
\end{lem}

\begin{proof}
A standard calculation yields
\begin{eqnarray*}
\mathbf{E}\left(\int_0^T Q_t^U d\langle M\rangle_t \right)^2&=&\mathbf{E}\left(\int_0^T g(t,T) U_t dt   \right)^2 \\
&=& \int_0^T \int_0^T g(s,T)g(t,T)\mathbf{E} (U_s U_t)dsdt\\
&\leq& \int_0^T e^{-2\beta(T-t)}dt+C_{H,\beta}H(2H-1)\int_0^T\int_0^T g(t,T)g(s,T)|t-s|^{2H-2}dsdt\\
&=&\int_0^T e^{-2\beta(T-t)}dt+C_{H,\beta}\int_0^T(1-g(s,T))g(s,T)ds\\
&=&\frac{1}{2\beta}\left(1-e^{-2\beta T}\right)+2C_{H,\beta}\langle M\rangle_T\,,
\end{eqnarray*}
which implies the desired result.
\end{proof}

\begin{lem}\label{limit Qt}
Let $H>1/2$, as $T\rightarrow \infty$, we have 
\begin{equation}\label{eq: convergence in prob}
\frac{1}{T}\int_0^TQ_t^2d\langle M\rangle_t \xrightarrow{\mathbf{P}}\frac{1}{2\beta}\,.
\end{equation}

Moreover, from the martingale convergence theorem, we have
\begin{equation}\label{eq: convergence in distribution}
\frac{1}{\sqrt{T}}\int_0^T Q_tdM_t\xrightarrow{d}\mathcal{N}\left(0,\frac{1}{2\beta}\right)\,.
\end{equation}
\end{lem}

\begin{proof}
From the definition of $Q_t$, we can write $Q_t^2$ as
\begin{eqnarray}\label{Eq: Q t 2}
Q_t^2&=&\left(\frac{\alpha}{\beta}-\frac{\alpha}{\beta}V(t)+Q_t^U\right)^2\notag\\
&=&\left(\frac{\alpha}{\beta}\right)^2+\left(\frac{\alpha}
{\beta}\right)^2V^2(t)+(Q_t^U)^2-2\left(\frac{\alpha}{\beta}\right)^2
V(t)+\frac{2\alpha}{\beta}Q_t^U-\frac{2\alpha}{\beta}V(t)Q_t^U\,.
\end{eqnarray}

Using \eqref{Eq: Q t 2}, can write our target quantity as
\begin{eqnarray}\label{Eq: 30}
\frac{1}{T}\int_0^TQ_t^2d\langle M\rangle_t&=&\frac{1}{T}\int_0^T \left(\frac{\alpha}{\beta}\right)^2d\langle M\rangle_t+
\frac{1}{T}\int_0^T\left(\frac{\alpha}{\beta}\right)^2V^2\left(t\right) d\langle M\rangle_t+\frac{1}{T}\int_0^T \left(Q_t^U\right)^2 d\langle M\rangle_t\notag\\
&&-2\frac{1}{T}\int_0^T \left(\frac{\alpha}{\beta}\right)^2V(t) d\langle M\rangle_t+\frac{1}{T}\int_0^T \left(\frac{2\alpha}{\beta}Q_t^U\right) d\langle M\rangle_t\notag\\
&&-\frac{1}{T}\int_0^T\left(\frac{2\alpha}{\beta}V(t)Q_t^U\right) d\langle M\rangle_t\,.
\end{eqnarray}

We consider the above six integrals separately. First, as $T\rightarrow\infty$, it is easy to see that
\begin{equation}\label{eq convergence constant}
\frac{1}{T}\int_0^T\left(\frac{\alpha}{\beta}\right)^2d\langle M\rangle_t=\frac{\langle M\rangle_T}{T}\xrightarrow{a.s.} 0\,.
\end{equation} 

Now, we deal with the second term in \eqref{Eq: 30}. A standard calculation implies
\begin{equation}\label{eq: dev Vt 2}
V^2(t)=\frac{d}{d\langle M\rangle_t}\int_0^t g(s,t)e^{-\beta s}ds
=\frac{1}{g^2(t,t)}\left(g(t,t)e^{-\beta t}+\int_0^t \dot{g}(s,t)e^{-\beta s}ds  \right)\,,
\end{equation}
where $\dot{g}(s,t)=\frac{\partial}{\partial t}g(s,t)$. With Cauchy-Schwarz inequality, we have
\begin{equation}\label{eq: dot g}
\int_0^t \dot{g}(s,t)e^{-\beta s}ds\leq \sqrt{\int_0^t \dot{g}^2(s,t)ds\int_0^t e^{-2\beta s}ds} \leq C\sqrt{\int_0^t s^{4H-4}ds \int_0^t e^{-2\beta s}ds}\,,
\end{equation}
where $C$ is a constant and the last inequality comes from the fact that
$$
\dot{g}(s,t)+H(2H-1)\int_0^t \dot{g}(r,t)|r-s|^{2H-2}=-H(2H-1)g(t,t)|s-t|^{2H-2},\,\, s\in (0,t),\, t>0.
$$

Combining \eqref{eq: dev Vt 2}, \eqref{eq: dot g} and Lemma \ref{quadratic variation big}, we can easily obtain
\begin{equation}\label{eq: convergence Vt 2}
\lim_{T\rightarrow \infty}\frac{1}{T}\int_0^T V^2(t)dt=0\,.
\end{equation}

From \cite{CK18}, as $T\rightarrow\infty$, we have
\begin{equation}\label{eq: convergence proba Qu}
\frac{1}{T}\int_0^T (Q_t^U)^2d\langle M\rangle_t \xrightarrow{\mathbf{P}}\frac{1}{2\beta}\,.
\end{equation}
 
Moreover, from Lemma \ref{limit V}, as $T\rightarrow\infty$, we have the following convergence
\begin{equation}\label{eq: convergence Vt as}
\frac{1}{T}\int_0^T V\left(t\right)d\langle M\rangle_t \xrightarrow{\mathbf{P}} 0\,.
\end{equation}

Next, from the proof of Lemma \ref{limit Q U} and Borel-Cantelli theorem, as $T\rightarrow\infty$, we obtain
\begin{equation}\label{eq: lim Qu almost}
\frac{1}{T}\int_0^T Q_t^U d\langle M\rangle_t\xrightarrow{a.s.} 0\,.
\end{equation}

With Cauchy-Schwarz inequality, \eqref{eq: convergence proba Qu} and \eqref{eq: convergence Vt 2}, we obtain
\begin{equation}\label{eq: convergece product}
\left|\frac{1}{T}\int_0^TV(t)Q_t^Ud\langle M\rangle_t\right|\leq \sqrt{\frac{1}{T}\int_0^T V^2(t)d\langle M\rangle_t\frac{1}{T}\int_0^T(Q_t^U)^2d\langle M\rangle_t}\xrightarrow{\mathbf{P}}0.
\end{equation}

Finally, the convergence in probability of \eqref{eq: convergence in prob} can be obtained by \eqref{eq convergence constant}, \eqref{eq: convergence Vt 2}, \eqref{eq: convergence proba Qu}, \eqref{eq: convergence Vt as}, \eqref{eq: lim Qu almost} and \eqref{eq: convergece product}.

For the convergence of \eqref{eq: convergence in distribution}, since the process $\int_0^t Q_s d\langle M\rangle, t\in [0,T]$ is a martingale and its quadratic variance is $\int_0^t Q_s^2 d\langle M\rangle_s,\, t\in [0,T]$, we can obtain \eqref{eq: convergence in distribution} by the martingale convergence theorem.
\end{proof}

The following are the results for $H<1/2$. When Lemma \ref{limit V} is also available for all $H\in (0,1)$, then 
\begin{lem}\label{V asymptotic smaller}
For $H<1/2$, we have 
$$
V(t)\sim O(1/t),\,\, t\rightarrow \infty.
$$
\end{lem}
\begin{proof}
$$
\int_0^T V(t) d\langle M\rangle_t=\int_0^T \left(V(t)\frac{d\langle M\rangle_t}{dt}\right)dt \sim Const.
$$
The result is clear with Lemma \ref{quadratic smaller}
\end{proof}
Now, we deal with the difficulty of the integral of $Q_t^U$:
\begin{lem}\label{lim Qt U Vt smaller}
For $H<1/2$ we have
$$
\frac{1}{T\langle M\rangle_T}\left(\int_0^T Q_t^Ud\langle M\rangle_t\right)^2=\frac{1}{T\langle M\rangle_T}\left(\int_0^T g(s,T)U_sds\right)^2\xrightarrow {a.s.}0
$$
and 
$$
\frac{2}{T}\left(\frac{\alpha}{\beta}\right)\int_0^T V(t)Q_t^Ud\langle M\rangle_t\xrightarrow {a.s.}0
$$
\end{lem}
\begin{proof}
$$
\int_0^T Q_t^Ud\langle M\rangle_t=\int_0^T \frac{d}{d\langle M\rangle_t}\int_0^t g(s,t)U_sds d\langle M\rangle_t=\int_0^T g(t,T) U_tdt
$$
because $U_t=\int_0^t e^{-\beta(t-s)}d\xi_s$, this integral can be regarded as the limit of Riemman sum from the theorem of Young integral \cite{Young36}, let us divide the interval [0,t] in to $2^n$ part and denote $0=t_0<t_1<\cdots<t_{2^n}=t$, then 
$$
U_t=\lim_{n\rightarrow \infty} \sum_{i=1}^{2^n} e^{-\beta(t-s_i)}(\xi_{t_i}-\xi_{t_{i-1}}),\, s_i\in (t_{i-1},t_i),
$$
With the H\"older continuity of the fractional Brownian motion and standard Brownian motion, there exists a constant $M$ such that 
$$
\big|\xi_{t_i}-\xi_{t_{i-1}}\big|\leq M|t_i-t_{i-1}|^{1/2-\varepsilon},\, i=1, 2,\cdots,\, 2^n
$$
for some $\varepsilon>0$. Then we have the inequality
$$
|U_t|\leq  M \int_0^t e^{-\beta(t-s)}ds^{1/2-\varepsilon}
$$
for some $\varepsilon>0$. From Lemma \ref{quadratic smaller} we can easily obtain 
$$
\frac{1}{\sqrt{T\langle M\rangle_T}}\int_0^T g(t,T) U_tdt  \xrightarrow {a.s.}0
$$
That is the first result. The second result is immediate when $V(t)\sim O(1/t)$ for $t$ large enough.
\end{proof}

\textbf{Acknowledge}: This work is supported by the National Nature Science Foundation of China (No. 71871202)

\end{document}